\documentclass[a4paper]{amsart}
\usepackage{amssymb}
\usepackage{amsfonts}

\newtheorem{theorem}{Theorem}
\theoremstyle{plain}

\newtheorem{definition}{Definition}

\newtheorem{lemma}{Lemma}

\numberwithin{equation}{section}
\input{tcilatex}

\begin{document}
\title{Projections in moduli spaces of Kleinian groups}
\author{Hala Alaqad}
\address{Department of Mathematical Sciences\\
United Arab Emirate Universty}
\email{hala\_a@uaeu.ac.ae}
\author{Jianhua Gong}
\address{Department of Mathematical Sciences\\
United Arab Emirate Universty}
\email{j.gong@uaeu.ac.ae}
\author{Gaven Martin}
\address{School of Natural and Computational Sciences\\
Massey University,}
\email{g.j.martin@massey.ac.nz}
\thanks{Research supported by UAEU UPAR Grant (31S315). This paper is partly
contained in the first author's PhD thesis.}

\begin{abstract}
A two-generator Kleinian group $\langle f,g \rangle$ can be naturally
associated with a discrete group $\langle f,\phi \rangle$ with the generator 
$\phi$ of order $2$ and where 
\begin{equation*}
\langle f,\phi f \phi^{-1} \rangle= \langle f,gfg^{-1} \rangle \subset
\langle f,g\rangle, \quad [ \langle f,g f g^{-1} \rangle: \langle f,\phi
\rangle]=2
\end{equation*}
This is useful in studying the geometry of Kleinian groups since $\langle
f,g \rangle$ will be discrete only if $\langle f,\phi \rangle$ is, and the
moduli space of groups $\langle f,\phi \rangle$ is one complex dimension
less. This gives a necessary condition in a simpler space to determine the
discreteness of $\langle f,g \rangle$.

The dimension reduction here is realised by a projection of principal
characters of two-generator Kleinian groups. In applications it is important
to know that the image of the moduli space of Kleinian groups under this
projection is closed and, among other results, we show how this follows from
J\o rgensen's results on algebraic convergence.
\end{abstract}

\maketitle

\section{Introduction}

It is shown in \cite{GehMar,GehMar2} that every two-generator Kleinian group 
$\left \langle f,\,g\right \rangle $ is determined uniquely up to conjugacy
by its \emph{principal character, }i.e.,\emph{\ }the triple of complex
parameters 
\begin{equation}
(\gamma ,\beta ,\tilde{\beta})=\left( \gamma (f,g),\beta (f),\beta
(g)\right) .  \label{character}
\end{equation}
Here 
\begin{equation}
\gamma(f,g)=\mathrm{tr}[f,g]-2, \quad \beta(f)=\mathrm{tr}^2(f)-4,\quad
\beta(g)=\mathrm{tr}^2(g)-4.
\end{equation}
Thus, the\textbf{\ }space of two-generator Kleinian groups $\left \langle
f,\,g\right \rangle $ can be identified as a subspace of the three complex
dimensional space $\mathbb{C}^{3}$. We define 
\begin{equation*}
\mathcal{D}=\left \{ (\gamma ,\beta ,\tilde{\beta}):(\gamma ,\beta ,\tilde{%
\beta})\text{ is the principal character of a Kleinian group}\right \}
\end{equation*}%
via the mapping 
\begin{equation*}
\left \langle f,\,g\right \rangle \longmapsto \left( \gamma \left(
f,\,g\right) ,\beta \left( f\right) ,\beta (g)\right) .
\end{equation*}
Things are set up so the group $\langle identity \rangle \longmapsto
(0,0,0)\in \mathbb{C}^3.$

Applying \cite[Lemma 2.29]{GehMar}, given a two-generator Kleinian group $%
\langle f,g\rangle $ with principal character $(\gamma ,\beta ,\tilde{\beta}%
) $ with $\gamma \neq \beta $, there is a two-generator discrete group $%
\Gamma _{\phi }=\langle f,\phi \rangle $ with principal character $(\gamma
,\beta ,-4)$, i.e., $\phi $ is elliptic of order $2.$ However, we can waive
the condition $\gamma \neq \beta $ (see Lemma \ref{2pars} below).

If the group $\Gamma_\phi$ is Kleinian (i.e., not virtually abelian), then
there are suites of inequalities on the entries of the principal character
which need to hold. One such is J\o rgensen' inequality, $%
|\gamma|+|\beta|\geq 1.$ Conversely, it is often easier to establish new
discreteness conditions of two generator groups with one generator of order
two since we know of many polynomial trace identities which yield a
polynomial semigroup action on the associated moduli space, and the dynamics
of this action can be used to investigate discreteness of groups. Let us
give one such example leading to J\o rgensen's inequality to motivate our
subsequent results. We note 
\begin{equation}
\mathrm{tr} [f,gfg^{-1}]-2 = (\mathrm{tr} [f,g]-2)(\mathrm{tr}
[f,gfg^{-1}]+2-\mathrm{tr}^2(f)),
\end{equation}
which we write as $\gamma(f,gfg^{-1}) =\gamma(\gamma-\beta)$ Following \cite%
{GMBull} we write $p_\beta(z)=z(z-\beta)$ and $p_\beta^{(n)}(z)=(p_\beta%
\circ p_\beta \circ \cdots \circ p_\beta)(z)$ then repeated application of
this identity followed by projection yields the sequence 
\begin{equation*}
(\gamma ,\beta ,-4) \mapsto (p_\beta(\gamma),\beta ,-4) \mapsto \cdots
\mapsto (p_\beta^{(n)}(\gamma) ,\beta ,-4) \mapsto \cdots
\end{equation*}
giving a sequence of traces of commutators $\{p_\beta^{(n)}(\gamma)%
\}_{n=0}^{\infty}$ in the original discrete group $\Gamma_\phi$. It is
usually not too difficult to prove such sequences cannot be infinite and
contain convergent subsequences. J\o rgensen's inequality is implied by the
assertion that $p_\beta^{(n)}(\gamma)\not \to0$, hence $|\gamma|>1+|\beta|$.
Every such trace identity yields a new inequality and the semigroup
structure is clearly seen in general by 
\begin{equation*}
(\gamma ,\beta ,-4) \mapsto (p_\beta(\gamma),\beta ,-4) \mapsto
(q_\beta(p_\beta(\gamma)),\beta ,-4)
\end{equation*}
where, for example, $q_\beta(z)=z(1+\beta-z)^2$ from the trace identity 
\begin{equation*}
\mathrm{tr}[f,gfg^{-1}fg]-2 = (\mathrm{tr} [f,g]-2)(\mathrm{tr}^2(f)-\mathrm{%
tr} [f,g]-1)^2, \; \; \gamma(f,gfg^{-1}fg)=\gamma(1+\beta-\gamma)^2
\end{equation*}
In another direction, the ergodicity of the polynomial semigroup acting in
the complement of the Riley slice, 
\begin{equation*}
{\mathcal{R}} =\mathrm{int}\{ (\gamma,0,-4): \langle f,g \rangle \; \; %
\mbox{is discrete and freely generated by $f$ and $g$} \}
\end{equation*}
is used to prove the supergroup density and the existence of unbounded
Nielsen classes of generating pairs, \cite{Mril}. Note that this is not the
usual definition of the Riley slice but is known to be equivalent.

In order to further these studies we establish a number of useful results in
this article. Dynamically, the elementary subgroups are sinks for the
polynomial semigroup action on a slice and are isolated within the subset
corresponding to discrete groups (apart from trivial abelian and dihedral
examples). Thus, after we observe elementary properties of a two-generator
discrete group with principal character $(\gamma ,\beta ,-4)$ in $\S 3$, we
list the possible principal characters in Table $1$ representing a
two-generator elementary discrete group with the principal character $\left(
\gamma ,\beta ,-4\right) $ in $\S 4$. Then we prove the two complex
dimensional slice space 
\begin{equation*}
\mathcal{D}_2 = \{(\gamma,\beta): 
\mbox{there is $\tilde{\beta}$ so
$(\gamma,\beta,\tilde{\beta})$ a principal character of a Kleinian group} \}
\end{equation*}
is closed in $\S 5$. Then $\mathcal{D}_2$ is the projection of the three
complex dimensional moduli space $\mathcal{D}$ to $\mathbb{C}^2$.
Applications of this set being closed are in the generalisations of J\o %
rgensen's inequality quantifying the isolated nature of the elementary
groups in the moduli space of discrete groups, \cite{AGM2}, \cite{AGM3}.

\section{Preliminaries}

Let $\mathrm{M\ddot{o}b}^{+}(\overline{\mathbb{C}})$ be the M\"{o}bius group
of the normalized orientation preserving M\"{o}bius transformations on the
Riemann sphere $\overline{\mathbb{C}}:$%
\begin{equation*}
\mathrm{M\ddot{o}b}^{+}(\overline{\mathbb{C}})\mathbb{=}\left \{ \frac{az+b}{%
cz+d}:a,b,c,d\in \mathbb{C} \text{ and}\ ad-bc=1\right \} ,
\end{equation*}
and let $\mathrm{PSL}(2, \mathbb{C} )$\ be the projective special linear
topological quotient group :%
\begin{equation*}
\mathrm{PSL}(2, \mathbb{C} )=\mathrm{SL}(2, \mathbb{C} )/\left \{ \pm
Id\right \} ,
\end{equation*}
where $\mathrm{SL}(2, \mathbb{C} )=\left \{ 
\begin{pmatrix}
a & b \\ 
c & d%
\end{pmatrix}
:a,b,c,d,\in \mathbb{C} \text{ and }ad-bc=1\right \} .$ Then these two
groups are topologically isomorphic, $\mathrm{M\ddot{o}b}^{+}(\overline{%
\mathbb{C}})\cong \mathrm{PSL}(2,\mathbb{C} )$ and so there are at least two
different ways of thinking about groups throughout this paper, as either
subgroups of $\  \mathrm{M\ddot{o}b}^{+}\left( \overline{\mathbb{C}}\right) $
or subgroups of $\  \mathrm{PSL}(2,\mathbb{C}).$ Each of these groups has its
own topology, however the topological isomorphism shows us that the concept
of discreteness is the same.

A subgroup $\Gamma $ of $\mathrm{M\ddot{o}b}^{+}(\overline{\mathbb{C}})$ is
called a \emph{Kleinian group} if it is a non-elementary discrete group (see 
\cite{Beardon,Maskit}). Here a group is \emph{elementary} if it is virtually
abelian, i.e. a finite extension of an abelian group. The elementary groups
are classified \cite[\S 5.1]{Beardon}. A subgroup of $\mathrm{M\ddot{o}b}%
^{+}(\overline{\mathbb{C}})$ is \emph{non-elementary} if the limit set, 
\begin{equation*}
L\left( \Gamma \right) =\text{the set of accumulation points of a generic
orbit } \Gamma \left( z\right) =\{f(z):f\in \Gamma \}
\end{equation*}%
contains infinitely many points. Notice from \cite{Jorg} that a group $%
\Gamma $ is Kleinian if and only if every two-generator subgroup of $\Gamma $
is discrete. It is this characterization of Kleinian groups motivates a
focus on two-generator groups.

The principal character of a representation of the two-generator group $%
\left \langle f,\,g\right \rangle $ conveniently encodes various other
geometric quantities. One of the most important is the \emph{complex
hyperbolic distance} $\delta +i\theta $ between two axes $\mathrm{ax}(f)$
and $\mathrm{ax}(g)$ of two elliptic or loxodromic M\"{o}bius
transformations $f$ and $g,$ where $\mathrm{ax}(f)$ is the hyperbolic line
in ${\mathbb{H}}^{3}$ joining the fixed points of $f$ and $g$ in $\overline{%
\mathbb{C}}$, $\delta $ is the hyperbolic distance between those two axes,
and $\theta $ is the holonomy of the transformation whose axis contains the
common perpendicular between those two axes and moves $\mathrm{ax}(f)$ to $%
\mathrm{ax}(g).$ Thus, 
\begin{equation}
\sinh^2(\delta+i\theta) = \frac{4\gamma}{\beta \tilde{\beta}}
\end{equation}

The parameter $\beta (f)$ describes the conjugacy class of M\"{o}bius
transformation in which $f$ lies.

\begin{lemma}
\label{elem by para}If an element $f$ of $\mathrm{M\ddot{o}b}^{+}(\overline{%
\mathbb{C}})$ is not the identity, then

$\left( a\right) $ $f$ is parabolic if and only if $\beta (f)=0.$

$\left( b\right) $ $f$ is elliptic if and only if $\beta (f)\in [-4,0).$

$\left( c\right) $ $f$ is loxodromic if and only if $\beta (f)\in \mathbb{C}%
\setminus \lbrack 0,4].$
\end{lemma}

In addition, the parameter $\beta (f)$ gives the order of each elliptic
element.

\begin{lemma}
\label{beta & order}If an element $f$ of $\mathrm{M\ddot{o}b}^{+}\left( 
\overline{\mathbb{C}}\right) $ is not the identity, then $f$ is elliptic of
order $p$ if and only if 
\begin{equation*}
\beta \left( f\right) \ =-4\sin ^{2}(\frac{k\pi }{p}),\text{ \ for }1\leq k<p%
\text{\ and }(k,p)=1.
\end{equation*}
\end{lemma}

We use the following easy lemma: a group is Kleinian if and only if it is 
\emph{virtually Kleinian, } i.e. has a Kleinian subgroup of finite index.

\begin{lemma}
\label{virtually Kleinian}Let $H$ be a finite index subgroup of $\Gamma $.
Then $\Gamma $ is Kleinian if and only if $H$ is Kleinian.
\end{lemma}

The proof follows directly from the fact that if $H$ has finite index in $%
\Gamma ,$ then it has a coset decomposition 
\begin{equation*}
\Gamma =\phi _{1}H\cup \phi _{2}H\cup \cdots \cup \phi _{n}H,
\end{equation*}%
where $\phi _{1}=Id,\phi _{2,}\cdots ,\phi _{n}\in \Gamma $ for some $n.$
Thus any sequence $\left \{ g_{n}\right \}$ contains a subsequence $\left \{
g_{n_{k}}\right \} $ lying in one coset.

\section{Slicing spaces of two-generator discrete groups}

Suppose that $\left \langle f,\,g\right \rangle $ is two-generator discrete\
subgroup of $\mathrm{M\ddot{o}b}^{+}(\overline{\mathbb{C}})$ with $\gamma
=\gamma \left( f,\,g\right) \neq 0$ and $\gamma \left( f,\,g\right) =\gamma
\neq \beta =\beta \left( f\right) ,$ Then Gehring and Martin \cite[Lemma 2.29%
]{GehMar} states that there exist elliptics $\phi $ and $\psi $ of order $2$
such that $\Gamma _{\phi }=\langle f,\phi \rangle $ and\ $\Gamma _{\psi
}=\langle f,\psi \rangle $ are discrete with the related two-generator
groups $\Gamma _{\phi }=\langle f,\phi \rangle $ and\ $\Gamma _{\psi
}=\langle f,\psi \rangle $ are discrete with principal characters $(\gamma
,\beta ,-4)$ and $(\beta -\gamma ,\beta ,-4)$, respectively.

Notice that $\gamma(f,g)=0$ implies $f$ and $g$ have a common fixed point on 
$\overline{\mathbb{C}}$ and hence is elementary (see \cite[Theorem 4.3.5]%
{Beardon} and \cite[Identity (1.5)]{GehMar}). Thus, the condition $\gamma
\neq 0$ for Kleinian groups. We now consider the other restriction $\gamma
\neq \beta $ of \cite[Lemma 2.29]{GehMar}.

\medskip

We start with our observation of the properties of a two-generator discrete
group with the principal character $\left( \gamma ,\beta ,-4\right) ,$ which
has one generator of order $2.$

\begin{lemma}
\label{gammabeta}Let $\Gamma =\left \langle f,\,g\right \rangle $ be a
discrete group with principal character $\left( \gamma ,\beta ,-4\right) .$
If $\gamma =\beta ,$ then $\Gamma $ is an elementary discrete group, which
is isomorphic to a dihedral group $D_{p},$ for some $p$ $\in \left \{
1,2,\cdots ,\infty \right \} $.
\end{lemma}

\begin{proof}
There are two cases to consider. If $f$ is parabolic, then $\gamma=\beta =0$
by Lemma \ref{elem by para} and $\Gamma $ as $f$ and $g$ have a common fixed
point which we may assume by conjugacy is $\infty \in \overline{\mathbb{C}}$%
. Then $\Gamma$ is a subgroup of the upper-triangular matrices 
\begin{equation*}
f=\left( 
\begin{array}{cc}
1 & x \\ 
0 & 1%
\end{array}%
\right) \text{ and } g=\left( 
\begin{array}{cc}
\pm i & y \\ 
0 & \mp i%
\end{array}
\right)
\end{equation*}
and the result follows. We may now assume $f$ and $g$ have the form 
\begin{equation*}
f=\left( 
\begin{array}{cc}
\lambda & 0 \\ 
0 & \frac{1}{\lambda }%
\end{array}%
\right) \text{ and }g=\left( 
\begin{array}{cc}
\alpha & \mu \\ 
\delta & \nu%
\end{array}%
\right) ,\text{ where }\lambda \neq 0,\pm 1,\alpha \nu -\mu \delta =1.
\end{equation*}%
Computing the commutator,%
\begin{equation*}
\left[ f,g\right] =\left( 
\begin{array}{cc}
\alpha \nu -\lambda ^{2}\mu \delta & \alpha \lambda ^{2}\mu -\alpha \mu \\ 
\frac{1}{\lambda ^{2}}\delta \nu -\delta \nu & \alpha \nu -\frac{1}{\lambda
^{2}}\mu \delta%
\end{array}%
\right) ,
\end{equation*}%
and hence $\gamma =\mathrm{tr}\left[ f,g\right] -2=-\mu \delta (\lambda -%
\frac{1}{\lambda })^{2},$ and $\beta =\mathrm{tr}^{2}(f)-4=(\lambda -\frac{1%
}{\lambda })^{2}$. The assumption $\gamma =\beta $ now implies that $\mu
\delta =-1$. Since $\alpha \nu -\mu \delta =1,$$\alpha \nu =0 $. Now $\beta
(g)=-4,$ so $\mathrm{tr}(g)=0,$ which gives $\alpha +\nu =0$. Then $\alpha
=\nu =0$ and 
\begin{equation*}
g=\left( 
\begin{array}{cc}
0 & \mu \\ 
-\frac{1}{\mu } & 0%
\end{array}%
\right) .
\end{equation*}%
Thus, $g(z)=-\frac{\mu ^{2}}{z}$ interchanges the fixed points $0$ and $%
\infty $ of $f.$ So $\Gamma $ has a finite orbit $\left \{ 0,\infty
\right
\} $\ and hence it is an elementary discrete group, \cite{Beardon}.
Computation shows the conjugate $gfg^{-1}=f^{-1}:$ 
\begin{equation*}
gfg^{-1}=\left( 
\begin{array}{cc}
\frac{1}{\lambda } & 0 \\ 
0 & \lambda%
\end{array}%
\right) =f^{-1}.
\end{equation*}%
Thus $\Gamma $ is a dihedral group $D_{p}$, for some $p \in \mathbb{N}\cup
\{ \infty \}$ by from Theorem \ref{classific of disc elem}.
\end{proof}

\begin{lemma}
\label{two order 2}Let $\Gamma =\left \langle f,\,g\right \rangle $ be a
discrete group with the principal character $\left( \gamma ,-4,-4\right) ,$
then $\Gamma $ is isomorphic to one of the following elementary discrete
groups:

$(a)$ The dihedral group $D_{\infty },$ if $\mathrm{ax}(f)\cap \mathrm{ax}%
(g)=\varnothing $ or $\mathrm{ax}(f)\cap \mathrm{ax}(g)\neq \varnothing $ in 
$%
\mathbb{C}
.$

$(b)$ A dihedral group $D_{p}$ for some $p\in 
\mathbb{N}
,$ if $\mathrm{ax}(f)\cap \mathrm{ax}(g)\neq \varnothing $ in ${\mathbb{H}}%
^{3}.$

$(c)$ A parabolic group, if $\mathrm{ax}(f)\cap \mathrm{ax}(g)=\left \{
\infty \right \} .$
\end{lemma}

\begin{proof}
$(1)$ $\mathrm{ax}(f)\cap $ $\mathrm{ax}(g)=\varnothing :$ Let $\alpha $ be
the common perpendicular between two axes $\mathrm{ax}(f)$ and $\mathrm{ax}%
(g).$ Since $f$ and $g$ are elliptic of order $2,$ they are the rotations of
order $2$ about their axes $\mathrm{ax}(f)$ and $\mathrm{ax}(g)$ in ${%
\mathbb{H}}^{3},$ respectively. And each of $f$ and $g$ interchanges the
ending points of $\alpha $ and fixes $\alpha $ setwise, hence $\Gamma $ is
elementary. Furthermore, the product $fg$ fixes the ending points of $\alpha 
$ and fixes $\alpha $ setwise, so the axis $\mathrm{ax}\left( fg\right)
=\alpha .$ It follows that $fg$ is loxodromic.

Next, from the Fricke identity (see \cite{Fricke-Klein}), we have 
\begin{equation}
\gamma \left( f,\,g\right) =\beta \left( f\right) +\beta (g)+\beta (fg)-%
\mathrm{tr}(f)\mathrm{tr}(g)\mathrm{tr}(fg)+8.  \label{Friche beta Id}
\end{equation}

Since $\beta \left( f\right) =\beta \left( g\right) =-4,$ $\mathrm{tr}(f)=0,$
and $f^{2}=Id,$ the previous identity (\ref{Friche beta Id}) gives $\gamma
(f,fg)=\beta (fg). $ and by Lemma \ref{gammabeta}, $\langle f,fg\rangle $ is
a dihedral group. Notice that $fg$ is loxodromic with infinite order, hence $%
\left \langle f,\,g\right \rangle =\langle f,fg\rangle \cong D_{\infty }.$

$(2)$ $\mathrm{ax}(f)\cap $ $\mathrm{ax}(g)\neq \varnothing $ in ${\mathbb{H}%
}^{3}:$ Since $f$ and $g$ have a common fixed point (say $x_{0}$), $\Gamma $
is elementary. Since $\Gamma $ is discrete , the axes of elliptics both of
order $2$ must intersect at $x_{0}$ an angle $\theta =\frac{k\pi }{n}$ for
some $k$ and $n\geq 2,$ the product $fg\ $is a rotation about $\mathrm{ax}%
(fg)$ with the oriented angle $2\theta =\frac{2k\pi }{p}$ for any $n\geq 2$
and some $p\in \mathbb{N} ,$ where $\mathrm{ax}(fg)$ is perpendicular to $%
\mathrm{ax}(f)$ and $\mathrm{ax}(g)$ and passing through $x_{0}.$ Therefore, 
$fg\ $is elliptic of order $p.$ Notice that $f$ and $g$ are rotations of
order $2$ and their rotation axes $\mathrm{ax}(f)$ and $\mathrm{ax}(g)$ meet
in $\overline{{\mathbb{H}}^{3}},$ we have $gfg^{-1}=f^{-1}.$ It follows that 
$\left \langle f,\,g\right \rangle $ $\cong D_{p}.$

$(3)$ $\mathrm{ax}(f)\cap $ $\mathrm{ax}(g)\neq \varnothing $ in $\overline{%
\mathbb{C}}:$ We may assume by conjugacy that $\infty$ is the common fixed
point of $f$ and $g$ in $\overline{\mathbb{C}} $. Then $f$ and $g$ fix $%
\infty ,$ and hence $\mathrm{ax}(f)$ and $\mathrm{ax}(g)$ are vertical
hyperbolic lines. Notice that both $f$ and $g$ are rotations with the angles 
$\pi $ at distinct centers. Since the sum of two rotation angles is $2\pi ,$
the product$\ fg$ is a translation and hence $fg$ is parabolic fixing $%
\infty .$ Thus the classification of elementary groups \cite[Section 5.1]%
{Beardon}, $\left \langle f,\,g\right \rangle $ is a group of Euclidean
similarities of $\mathbb{C}$.
\end{proof}

We next note the following elementary result.

\begin{lemma}
\label{index 2}Let $\Gamma =\left \langle f,\,g\right \rangle $ be a
two-generator group with principal character $\left( \gamma ,\beta
,-4\right) ,$ then $\Gamma ^{g}=\left \langle f,\text{\thinspace }%
gfg^{-1}\right \rangle $ is an index two subgroup of $\left \langle
f,\,g\right \rangle :$ $\Gamma =\Gamma ^{g}\cup g\Gamma ^{g}$.
\end{lemma}

Then we have the following.

\begin{lemma}
\label{fixedpoints}Suppose that $\left \langle f,\,g\right \rangle $ is a
two-generator Kleinian group with principal character $(\gamma ,\beta ,%
\widetilde{\beta }).$ If $\widetilde{\beta }\neq -4,$ then%
\begin{equation*}
\mathrm{Fix}(f)\cap \mathrm{Fix}\left( gfg^{-1}\right) =\varnothing .
\end{equation*}
\end{lemma}

\begin{proof}
Since $\left \langle f,\,g\right \rangle $ is a two-generator Kleinian
group, $\gamma \neq 0$ and hence $f$ and $g$ cannot share any fixed points,
so $\mathrm{Fix}(f)\cap \mathrm{Fix}(g)=\varnothing .$ Since $g(\mathrm{Fix}%
\left( f\right)) = \mathrm{Fix}(gfg^{-1})$, there are two cases for us to
consider.

If $\mathrm{Fix}(f)\ $has one element, say $\mathrm{Fix}(f)=\left \{
z_{0}\right \} $, then $\mathrm{Fix}\left( gfg^{-1}\right) =\{g(z_{0})\}.$
But $g(z_{0})\neq z_{0}.$ Thus, $\mathrm{Fix}(f)\cap \mathrm{Fix}\left(
gfg^{-1}\right) =\varnothing .$

If $\mathrm{Fix}(f)\ $has two elements, say $\mathrm{Fix}(f)=$ $\left \{
z_{1},z_{2}\right \} $, then $\mathrm{Fix}\left( gfg^{-1}\right)
=\{g(z_{1}),g(z_{2})\}$ and 
\begin{equation}
g(z_{1})\neq z_{1},g(z_{2})\neq z_{2}.  \label{fixed pts of f not of g}
\end{equation}%
If $\mathrm{Fix}(f)\cap \mathrm{Fix}\left( gfg^{-1}\right) \neq \emptyset ,$
then by \cite[Theorem 5.1.2]{Beardon} $f$ and $gfg^{-1}$cannot share exactly
one fixed point, and hence $\mathrm{Fix}(f)=\mathrm{Fix}\left(
gfg^{-1}\right) =g(\mathrm{Fix}(f)).$ Thus, the only possibility is that $g$
interchange the fixed points of $f$, so $g(z_{1})=z_{2}$ and $%
g(z_{2})=z_{1}, $ which implies that $g^{2}$ fixes $z_{1}$and $z_{2}$ and
are not fixed by $g.$ We conclude that $g^{2}$ has at least three fixed
points and this implies that $g^{2}$ is the identity, $g$ has order $2.$and $%
\tilde{\beta}=-4.$
\end{proof}

In \cite[Remark 2,6]{GM00} it is shown that if $\left \langle
f,\,g\right
\rangle $ is a two-generator discrete group, then there are two
elliptic elements $\phi $ and $\psi $ of order $2$ such that 
\begin{equation}
\varphi f\phi ^{-1}=gfg^{-1}\text{ \ and \ }\psi f\psi ^{-1}=gf^{-1}g^{-1}.
\label{two ellipt}
\end{equation}%
In particular, if $f$ is elliptic or loxodromic, then the axes $\mathrm{ax}%
\left( \phi \right) $ and $\mathrm{ax}\left( \psi \right) $ meet at
right-angles to one and other, and they bisect the common perpendicular
between the axes $\mathrm{ax}\left( f\right) $ and $\mathrm{ax}\left(
gfg^{-1}\right) .$

\begin{lemma}
\label{2pars}Let $\left \langle f,\,g\right \rangle $ be a Kleinian group
with principal character $(\gamma ,\beta ,\widetilde{\beta }).$ Then there
are two elliptic elements $\phi $ and $\psi $ of order $2$ such that $\Gamma
_{\phi }=\langle f,\phi \rangle $ and\ $\Gamma _{\psi }=\langle f,\psi
\rangle $ are discrete $\mathbb{Z}_{2}$-extension of $\Gamma ^{g}$ with
principal characters $(\gamma ,\beta ,-4)$ and $(\beta -\gamma ,\beta ,-4)$,
respectively.
\end{lemma}

\begin{proof}
Note $\Gamma ^{g}=\langle f,gfg^{-1}\rangle $ is a subgroup of both $\Gamma
_{\phi }=\langle f,\phi \rangle $ and\ $\Gamma _{\psi }=\langle f,\psi
\rangle .$ Suppose that the principal characters for $\Gamma _{\phi }$ and $%
\Gamma _{\psi }$ are $(\gamma _{1},\beta ,-4)$ and $(\gamma _{2},\beta ,-4),$
respectively. Using \cite[Lemma 2.1]{GehMar} 
\begin{align*}
\gamma (\gamma -\beta )& =\gamma (f,gfg^{-1})=\gamma (f,\phi f\phi
^{-1})=\gamma _{1}(\gamma _{1}-\beta ), \\
\gamma (\gamma -\beta )& =\gamma (f,gfg^{-1})=\gamma (f,\psi f\psi
^{-1})=\gamma _{2}(\gamma _{2}-\beta ).
\end{align*}%
we have two quadratic equations 
\begin{align*}
\gamma _{1}^{2}-\beta \gamma _{1}-\gamma (\gamma -\beta )& =0, \\
\gamma _{2}^{2}-\beta \gamma _{2}-\gamma (\gamma -\beta )& =0.
\end{align*}
Solving these two quadratic equations, $\gamma _{1}=\gamma ,$ or $\beta
-\gamma $ and $\gamma _{2}=\gamma ,$ or $\beta -\gamma .$ As \cite[Identity
(6.9)]{Martin3} shows both possibilities occur, so%
\begin{equation*}
\{ \gamma _{1},\gamma _{2}\}=\{ \gamma ,\beta -\gamma \}.
\end{equation*}
Thus, after relabeling, the principal characters for $\Gamma _{\phi }$ and $%
\Gamma _{\psi }$ are $(\gamma ,\beta ,-4)$ and $(\beta -\gamma ,\beta ,-4),$
respectively. This completes the proof.
\end{proof}

\section{ Exceptional set of principal characters.}

Now the natural question is to determine whether two discrete groups $\Gamma
_{\phi }=\langle f,\phi \rangle $ and\ $\Gamma _{\psi }=\langle f,\psi
\rangle $ in Lemma \ref{2pars} are actually Kleinian. Since both $\Gamma
_{\phi }$ and\ $\Gamma _{\psi }$ are $\mathbb{Z}_{2}$-extension of $\Gamma
^{g},$ by Lemma \ref{virtually Kleinian}, we only need to decide if $\Gamma
^{g}$ is a Kleinian subgroup. First, we list the possible principal
characters in Tables $1,2,$ and $3$ representing two-generator elementary
discrete groups. The set of principal characters occurring in these three
tables is called the \emph{exceptional set of principal characters. } We
recall from \cite[ Section 5.1]{Beardon}.

\begin{theorem}[Classification of elementary discrete groups]
\label{classific of disc elem}Let $G$ be an elementary discrete group of $%
\mathrm{M\ddot{o}b}^{+}(\overline{\mathbb{C}})$, then $G$ is isomorphic to
one of the following groups, where $p$ $\in \left \{ 1,2,\cdots ,\infty
\right \} .$\newline
$(a)$ A cyclic group ${\mathbb{Z}}_{p}:$%
\begin{equation*}
\langle a:a^{p}=Id\rangle .
\end{equation*}
$(b)$ A dihedral group $D_{p}\cong {\mathbb{Z}}_{p}\ltimes {\mathbb{Z}}_{2}:$%
\begin{equation*}
\langle a,b:a^{p}=b^{2}=Id,\text{ }bab^{-1}=a^{-1}\rangle .
\end{equation*}
$(c)$ The group $({\mathbb{Z}}_{p}\times {\mathbb{Z}})\ltimes {\mathbb{Z}}%
_{2}$ or ${\mathbb{Z}}_{p}\times {\mathbb{Z}}:$%
\begin{equation*}
\langle a,b,c:aba^{-1}b^{-1}=b^{p}=c^{2}=Id,\text{ }cac^{-1}=a^{-1},\text{ }%
cbc^{-1}=b^{-1}\rangle .
\end{equation*}
$(d)$ A Euclidean translation group ${\mathbb{Z}}\times {\mathbb{Z}}.$%
\newline
$(e)$ A Euclidean triangle group $\Delta (2,3,6),$ or $\Delta (2,4,4),$ or $%
\Delta (3,3,3):$%
\begin{equation*}
\langle a,b:a^{p}=b^{q}=(ab)^{r}=Id,\text{ }\frac{1}{p}+\frac{1}{q}+\frac{1}{%
r}=1\rangle .
\end{equation*}
$(f)$ A finite spherical triangle group $A_{4}=\Delta (2,3,3)$, or $%
S_{4}=\Delta (2,3,4),$ or $A_{5}=\Delta (2,3,5):$%
\begin{equation*}
\langle a,b:a^{p}=b^{q}=(ab)^{r}=Id,\text{ }\frac{1}{p}+\frac{1}{q}+\frac{1}{%
r}>1\rangle .
\end{equation*}
\end{theorem}

Let $f$ and $g$ be elliptic of order $p\in \left \{ 2,3,4,5\right \} $ in $%
\mathrm{M\ddot{o}b}^{+}(\overline{\mathbb{C}}),$ and let $\theta $ be the
angle subtended at the origin between $\mathrm{ax}(f)$ and $\mathrm{ax}(g)$
and hence the hyperbolic distance between those two axes is $\delta =0.$
Then, $\sin ^{2}(\theta )$ can be computed by using \cite[Lemmas 6.19, 6.20,
and 6.21]{Martin3}, the parameters $\beta(g) $ and $\beta ( f) \in \left \{
-3,-4,\frac{\sqrt{5}-5}{2},-\frac{\sqrt{5}+5}{2}\right \} $, and the
commutator parameter is $\gamma \left( f,\,g\right) =\frac{-\beta \left(
f\right) \beta (g)}{4}\sin ^{2}(\theta )$ by \cite[Identity (2.7)]{GM1}.
With these elementary observations, some spherical trigonometry, and Theorem %
\ref{classific of disc elem} about the classification of the elementary
discrete groups, we obtain the following Tables $1$ and $2$, which list the
possible principal characters (including the cases in the thesis \cite{Zhang}%
) of two-generator elementary discrete groups with non-zero commutator
parameters $\gamma (f,g)\neq 0.$ They are either the dihedral groups (if $%
\gamma \left( f,\,g\right) =\beta \left( f\right) $ and $\beta (g)=-4)$ or
the finite spherical triangle groups $A_{4}=\Delta (2,3,3)$, $S_{4}=\Delta
(2,3,4),$ and $A_{5}=\Delta (2,3,5).$

\begin{center}
\begin{equation*}
\ 
\begin{tabular}{|c|c|c|c|c|}
\multicolumn{5}{c}{Table $1.$\ Principal Characters$:2,\;p$.} \\ \hline
$p$ & $\sin ^{2}(\theta )$ & $\gamma $ & groups & principal characters \\ 
\hline
$3$ & $\frac{2}{3}$ & $-2$ & $A_{4}$ & $\left( -2,-3,-4\right) $ \\ \hline
$3$ & $\frac{1}{3}$ & $-1$ & $S_{4}$ & $\left( -1,-3,-4\right) $ \\ \hline
$3$ & $\frac{3-\sqrt{5}}{6}$ & $\frac{\sqrt{5}-3}{2}$ & $A_{5}$ & $\left( 
\frac{\sqrt{5}-3}{2},-3,-4\right) $ \\ \hline
$3$ & $\frac{3+\sqrt{5}}{6}$ & $-\frac{3+\sqrt{5}}{2}$ & $A_{5}$ & $\left( -%
\frac{3+\sqrt{5}}{2},-3,-4\right) $ \\ \hline
$3$ & $1$ & $-3$ & $D_{3}$ & $\left( -3,-3,-4\right) $ \\ \hline
$4$ & $\frac{1}{2}$ & $-1$ & $S_{4}$ & $\left( -1,-2,-4\right) $ \\ \hline
$4$ & $1$ & $-2$ & $D_{4}$ & $\left( -2,-2,-4\right) $ \\ \hline
$5$ & $\frac{5-\sqrt{5}}{10}$ & $\frac{\sqrt{5}-3}{2}$ & $A_{5}$ & $\left( 
\frac{\sqrt{5}-3}{2},\frac{\sqrt{5}-5}{2},-4\right) $ \\ \hline
$5$ & $\frac{5-\sqrt{5}}{10}$ & $-1$ & $A_{5}$ & $\left( -1,-\frac{5+\sqrt{5}%
}{2},-4\right) $ \\ \hline
$5$ & $\frac{5+\sqrt{5}}{10}$ & $-1$ & $A_{5}$ & $\left( -1,\frac{\sqrt{5}-5%
}{2},-4\right) $ \\ \hline
$5$ & $\frac{5+\sqrt{5}}{10}$ & $-\frac{3+\sqrt{5}}{2}$ & $A_{5}$ & $\left( -%
\frac{3+\sqrt{5}}{2},-\frac{5+\sqrt{5}}{2},-4\right) $ \\ \hline
$5$ & $1$ & $\frac{\sqrt{5}-5}{2}$ & $D_{5}$ & $\left( \frac{\sqrt{5}-5}{2},%
\frac{\sqrt{5}-5}{2},-4\right) $ \\ \hline
$5$ & $1$ & $-\frac{5+\sqrt{5}}{2}$ & $D_{5}$ & $\left( -\frac{5+\sqrt{5}}{2}%
,-\frac{5+\sqrt{5}}{2},-4\right) $ \\ \hline
\end{tabular}%
\end{equation*}%
$\ $%
\begin{equation*}
\begin{tabular}{|c|c|c|c|c|}
\multicolumn{5}{c}{Table $2.$\ Principal Characters$:$ $3,\;p$.} \\ \hline
$p$ & $\sin ^{2}(\theta )$ & $\gamma $ & groups & principal characters \\ 
\hline
$3$ & $\frac{4}{9}$ & $-1$ & $A_{5}$ & $\left( -1,-3,-3\right) $ \\ \hline
$3$ & $\frac{8}{9}$ & $-2$ & $A_{4}$ & $\left( -2,-3,-3\right) $ \\ \hline
$4$ & $\frac{2}{3}$ & $-1$ & $S_{4}$ & $\left( -1,-2,-3\right) $ \\ \hline
$5$ & $\frac{10-2\sqrt{5}}{15}$ & $\frac{\sqrt{5}-3}{2}$ & $A_{5}$ & $\left( 
\frac{\sqrt{5}-3}{2},\frac{\sqrt{5}-5}{2},-3\right) $ \\ \hline
$5$ & $\frac{10-2\sqrt{5}}{15}$ & $-1$ & $A_{5}$ & $\left( -1,-\frac{5+\sqrt{%
5}}{2},-3\right) $ \\ \hline
$5$ & $\frac{10+2\sqrt{5}}{15}$ & $-1$ & $A_{5}$ & $\left( -1,\frac{\sqrt{5}%
-5}{2},-3\right) $ \\ \hline
$5$ & $\frac{10+2\sqrt{5}}{15}$ & $-\frac{3+\sqrt{5}}{2}$ & $A_{5}$ & $%
\left( -\frac{3+\sqrt{5}}{2},-\frac{5+\sqrt{5}}{2},-3\right) $ \\ \hline
\end{tabular}%
\end{equation*}
\end{center}

Notice that the angle between intersecting axes of elliptics of order $4$ in
a discrete group is always either $0$ when they meet on the Riemann sphere $%
\overline{\mathbb{C}}$ or $\frac{\pi }{2}.$ This yields the additional
parameter $\left( -1,-2,-2\right) $ for the elementary group $S_{4}$ when
generated by two elements of order $4$. Furthermore, the angle between
intersecting axes of elliptics of order $5$ in a discrete group is either $%
\arcsin \frac{2}{\sqrt{5}}$ or its complement $\arcsin \frac{-2}{\sqrt{5}}$.
After possibly taking powers of the generator of order $5$, the three
additional principal characters can be obtained in the following Table $3$.

\begin{center}
$%
\begin{array}{c}
\begin{tabular}{|c|c|c|c|c|}
\multicolumn{5}{c}{Table $3.$\ Principal Characters$:4,4$ and $5,5$.} \\ 
\hline
$p,p$ & $\sin ^{2}(\theta )$ & $\gamma $ & Group & Parameters \\ \hline
$4,4$ & $1$ & $-1$ & $S_{4}$ & $\left( -1,-2,-2\right) $ \\ \hline
$5,5$ & $\frac{4}{5}$ & $\frac{\sqrt{5}-3}{2}$ & $A_{5}$ & $\left( \frac{%
\sqrt{5}-3}{2},\frac{\sqrt{5}-5}{2},\frac{\sqrt{5}-5}{2}\right) $ \\ \hline
$5,5$ & $\frac{4}{5}$ & $-1$ & $A_{5}$ & $\left( -1,-\frac{5+\sqrt{5}}{2},%
\frac{\sqrt{5}-5}{2}\right) $ \\ \hline
$5,5$ & $\frac{4}{5}$ & $-\frac{3+\sqrt{5}}{2}$ & $A_{5}$ & $\left( -\frac{3+%
\sqrt{5}}{2},-\frac{5+\sqrt{5}}{2},-\frac{5+\sqrt{5}}{2}\right) $ \\ \hline
\end{tabular}%
\end{array}%
\ $
\end{center}

The classification of the elementary discrete groups (Theorem \ref{classific
of disc elem} ) plays an important role in listing the tables above. We make
the following additional remarks.

\begin{itemize}
\item The axes of elliptics both of order $2$ can intersect at an angle $%
\frac{k\pi }{p}$ for any $k$ and $p\geq 2$ giving the dihedral group $D_{p}$
with principal character $\left( -4\sin ^{2}\frac{k\pi }{p},-4,-4\right) $
(see Lemma \ref{two order 2}).

\item The axes of elliptics elements of orders $2$ and $p$, $p\geq 3$, can
meet at right-angles. In this case, the principal character of dihedral
group $D_{p}$ is $\left( -4\sin ^{2}\frac{\pi }{p},-4\sin ^{2}\frac{\pi }{p}%
,-4\right) .$

\item The axes of elliptics of order $p$ and $q$ ( $p\leq q$) in a discrete
group meet on $\overline{\mathbb{C}}$, i.e., meeting with angle $0$, if and
only if 
\begin{equation*}
(p,q)\in \{(2,2),(2,3),(2,4),(2,6),(3,3),(3,6),(4,4),(6,6)\}
\end{equation*}%
For all of these groups $\gamma \left( f,\,g\right) =0.$ In particular,
Euclidean triangle groups $\Delta (2,3,6)$, $\Delta (3,3,3)$, $\Delta
(2,4,4),$ and cyclic groups have $\gamma \left( f,\,g\right) =0.$ Also,
Euclidean translation groups have $\gamma =0.$
\end{itemize}

\begin{lemma}
Let $\left \langle f,\,g\right \rangle $ be a Kleinian group with the
principal character $\left( \gamma ,\beta ,-4\right) $ such that is not one
of those exceptional groups listed in Table $1.$ Then the subgroup $\Gamma
^{g}=\left \langle f,\,gfg^{-1}\right \rangle $ is Kleinian.
\end{lemma}

\begin{proof}
We need to show that the discrete subgroup $\Gamma ^{g}$ of Kleinian group $%
\left \langle f,\,g\right \rangle $ can not be elementary. Let $\gamma
=\gamma \left( f,\,g\right) $ and $\beta =\beta \left( f\right) .$ By Lemma %
\ref{2pars} there is an elliptic $h$ of order $2$ such that $\Gamma
_{h}=\langle f,h\rangle $ is discrete group containing $\Gamma ^{g}$ with
index two and the principal character for $\Gamma _{h}$ is $(\gamma ,\beta
,-4)$. By the hypothesis, $\left( \gamma ,\beta ,-4\right) $ is not one of
those exceptional groups listed in Table $1,$ then $\Gamma _{h}$ is not a
finite spherical triangle group $A_{4},S_{4},$ and $A_{5}.$ Thus, $\Gamma
^{g}$ can only be elementary if $\Gamma ^{g}=\gamma \left( \gamma -\beta
\right) =0,$ i.e., $\gamma =0$ or $\gamma =\beta .$ Since $\left \langle
f,\,g\right \rangle $ is Kleinian, $\gamma \neq 0.$ So it can only be $%
\gamma =\beta \ $and hence it is the dihedral group. In this case, since $f$
is not of order $2$, $gfg^{-1}=f^{\pm 1}$. In the case $gfg^{-1}=f,$ it
gives $g\left( \mathrm{Fix}(f)\right) =\mathrm{Fix}\left( gfg^{-1}\right) =%
\mathrm{Fix}(f)$ and hence $g$ is elliptic of order $2$ (Lemma \ref%
{fixedpoints}) fixing or interchanging the fixed points of $f$ in $\overline{%
\mathbb{C}}$. In the case $gfg^{-1}=f^{-1},$ $g$ might be a power of $f$. In
either case $\left \langle f,\,g\right \rangle $ is not Kleinian, a
contradiction.
\end{proof}

The following theorem guarantees that $\Gamma ^{g}$ is a Kleinian group if $%
f $ is not elliptic of order $p\leq 6.$

\begin{theorem}
\label{Th: Ksubgp}Suppose that $\left \langle f,\,g\right \rangle $ is a
Kleinian group. If $f$ is not elliptic of order $p\in \{2,3,4,6\}$, then the
subgroup $\Gamma ^{g}=\left \langle f,\,gfg^{-1}\right \rangle $ is Kleinian.
\end{theorem}

\begin{proof}
$(a)$ Suppose that $g$ is elliptic of order $2.$ Then $\Gamma ^{g}$ is a
subgroup of index two in $\Gamma$ and Lemma \ref{virtually Kleinian} ensures
that $\Gamma ^{g}$ is also Kleinian. $(b)$ Suppose that the order of $g$ is
not $2$. We need to show that $\Gamma ^{g}$ is non-elementary. If $\Gamma
^{g}$ is elementary, then (see \cite[Section 5.1]{Beardon}). $(i)$ $\Gamma
^{g}$ is elementary of type I; each non-trivial element is elliptic. Since
the order of $f$ is not $p\leq 6$ or $p=5$, neither is the order of $%
gfg^{-1} $. If $\gamma \left( f,gfg^{-1}\right) \neq 0,$ then from Tables $%
1,2,$ and $3,$ $\Gamma ^{g}$ is neither the finite spherical triangle groups 
$A_{4},S_{4},$ and $A_{5},$ nor the dihedral groups $D_{3},D_{4},$ and $D_{5}
$. If $\gamma \left( f,gfg^{-1}\right) =0$, then $\gamma=\beta$, and $%
\mathrm{Fix}(f)\cap \mathrm{Fix}\left( gfg^{-1}\right) \neq \emptyset $. On
the other hand, by Lemma \ref{fixedpoints}, $\mathrm{Fix}(f)\cap \mathrm{Fix}%
\left( gfg^{-1}\right) =\emptyset$, a contradiction.

$(ii)$ Suppose $\Gamma ^{g}$ is an elementary group of type II; then $\Gamma
^{g}$ is conjugate to a subgroup of $\mathrm{M\ddot{o}b}(\overline{\mathbb{C}%
})$ fixing $\infty $ whose every element is parabolic of the form $z\mapsto
z+b$, $b\in \mathbb{C}$. Thus, the group $\Gamma ^{g}$ is abelian and hence $%
g\left( \mathrm{Fix}(f)\right) =\mathrm{Fix}\left( gfg^{-1}\right) =\mathrm{%
Fix}(f),$ a contradiction to $\Gamma$ being Kleinian.

$(iii)$ Suppose $\Gamma ^{g}$ is an elementary group of type III; then both $%
f$ and $gfg^{-1}$ are elliptic or both are loxodromic. In either case $%
\Gamma ^{g}$ is abelian and as above this is a contradiction.

We have shown that $\Gamma ^{g}$ cannot be elementary if $g$ does not have
order $2$. Hence in all two cases $\Gamma ^{g}$ is a Kleinian group.
\end{proof}

Notice that a Kleinian group $\langle f,g \rangle$ with $f$ elliptic of
order $p=5$ or $p\geq7$ and $\beta(g)\neq -4$, cannot have $\gamma=\beta$.
If so, then $f$ and $gfg^{-1}$ have a common fixed point. But this discrete
group cannot be a Euclidean group as they do not have elliptics of order $5$
or greater than $6$. Thus $f$ and $gfg^{-1}$ have two common fixed points
and the same trace. Hence $gfg^{-1}=f^{\pm 1}$ and $g$ must fix or
interchange the fixed points of $f$. The first is not possible, and the
second shows $g^2=identity$.

\section{Projections and principal characters}

Our methods to establish new inequalities rely on a fundamental result
concerning spaces of finitely generated Kleinian groups. Namely that they
are closed in the topology of algebraic convergence -- a result originally
due to J\o rgensen \cite{Jorg}. We state this now.

\begin{definition}
A sequence of $n$-generator subgroups $\Gamma _{j}=\left \langle
f_{j,1},f_{j,2},\cdots ,f_{j,n}\right \rangle $ is said to be \emph{%
convergent algebraically} to a $n$-generator subgroup $\Gamma =\langle
f_{1},f_{2},\cdots ,f_{n}\rangle $ in $\mathrm{M\ddot{o}b}(\overline{\mathbb{%
C}})$ if for each $k=1,2,\cdots ,n$, the sequence of the corresponding
generators$\left \{ f_{j,k}\right \} $ converges uniformly to $f_{k}$ in the
spherical metric of $\overline{\mathbb{C}}.$
\end{definition}

\begin{theorem}[J\o rgensen]
\label{Th: Jorgenson} If a sequence of $n$-generator Kleinian subgroups $%
\Gamma _{j}=\left \langle f_{j,1},f_{j,2},\cdots ,f_{j,n}\right \rangle $
converges algebraically to $\Gamma =\langle f_{1},f_{2},\cdots ,f_{n}\rangle 
$ in $\mathrm{M\ddot{o}b}(\overline{\mathbb{C}})$, then $\Gamma $ is a
Kleinian group.

Moreover, the map back is an eventual homomorphism. That is for all
sufficiently large $j$ the map\ $\Gamma \rightarrow \Gamma _{j}$ defined by $%
f_{k}\longmapsto f_{j,k}$ extends to a homomorphism of the groups.
\end{theorem}

There is a higher dimensional analogue to this J\o rgensen's theorem, but
additional hypotheses are necessary -- for instance giving a uniform bound
on the torsion in a sequence, or asking that the limit set be in geometric
position (see \cite[Proposition 5.8]{Martin}).

\bigskip

Now we show the slice $\mathcal{D}_2$ is closed. Notice that we are free to
normalize a sequence of two-generator groups by conjugation and to pass to a
subsequences. First a preliminary result.

\begin{theorem}
\label{D* closed}The slice%
\begin{equation*}
\mathcal{D}^{\ast }=\left \{ \left( \gamma ,\beta ,-4\right) :\left( \gamma
,\beta ,-4\right) \text{ is the principal character of a Kleinian group}%
\right \}
\end{equation*}%
\ is closed in the two complex dimensional space $\mathbb{C} ^{2}$ in the
usual topology.
\end{theorem}

\begin{proof}
Suppose that $\left \{ (\gamma _{j},\beta _{j},-4)\right \} $ is a sequence
of principal characters of Kleinian groups, say $\{ \left
\langle
f_{_{j}},g_{_{j}}\right \rangle \}$. We proceed by considering two cases: up
to subsequence, $f_{i}$ is parabolic for all $i$ or not.

\medskip

\noindent \textbf{Case (a).} Suppose $f_{_{j}}$ is parabolic for all $j$.
Conjugate each $\Gamma _{j}$ so that the first generator is now represented
by the matrix $f_{_{j}}=\Big( 
\begin{array}{cc}
1 & 1 \\ 
0 & 1%
\end{array}%
\Big)$. Then $f_{_{j}}$ is constant and it remains to show that the sequence 
$\{g_{_{j}}\}$ also converges. Since $tr(g_{_{j}})=0$, we suppose the matrix
for second generator is 
\begin{equation*}
g_{_{j}}=\left( 
\begin{array}{cc}
a_{j} & b_{_{j}} \\ 
c_{j} & -a_{j}%
\end{array}%
\right) .
\end{equation*}

We calculate that $\gamma_j=c_j^2$ and hence $c_j\to c$. J\o rgensen's
inequality gives $|\gamma_j|\geq 1$ for all $j$ and hence $c_j\to c \neq 0$.

Set 
\begin{equation*}
h_{_{j}}=\left( 
\begin{array}{cc}
1 & \frac{-a_{j}+i\sqrt{1+c_{_{j}}}}{c_{_{j}}} \\ 
0 & 1%
\end{array}%
\right) \text{ }\  \text{and \ }h_{_{j}}g_{_{j}}h_{_{j}}^{-1}=\left( 
\begin{array}{cc}
i\sqrt{1+c_{_{j}}} & 1 \\ 
c_{i} & -i\sqrt{1+c_{_{j}}}%
\end{array}%
\right) .
\end{equation*}
Since $f_{_{j}}$ commutes with $h_{j},$ $f_{_{j}}$ is left unchanged under
conjugation. The result follows by Theorem \ref{Th: Jorgenson}.

\medskip

\noindent \textbf{Case (b).} We now suppose that $f_{j}$ is not parabolic
for all $j$. By conjugation, we may assume that $f_{_{j}}$ is represented by
the matrix 
\begin{equation*}
f_{_{j}}=\left( 
\begin{array}{cc}
\lambda _{_{j}} & 0 \\ 
0 & \frac{1}{\lambda _{_{j}}}%
\end{array}%
\right) ,
\end{equation*}%
where $\lambda _{j}\neq \pm 1.$ Thus the parameter $\beta _{j}=(\lambda
_{_{j}}-\frac{1}{\lambda _{j}})^{2},$ which gives the quadratic equation 
\begin{equation*}
\lambda _{_{j}}^{2}-\sqrt{\beta _{j}}\lambda _{_{j}}-1=0
\end{equation*}
and hence $\lambda _{_{j}}=\frac{\sqrt{\beta _{j}}\pm \sqrt{\beta _{j}+4}}{2}%
.$ Since $\beta _{j}$ $\rightarrow \beta ,$ $\lambda _{i}$ converges to $%
\lambda =\frac{\sqrt{\beta }\pm \sqrt{\beta +4}}{2}.$ Noting that $\lambda
\neq 0,\pm 1$ for either choice of $\lambda $, we conclude that $f_{j}$
converges to the element $f$ represented by the matrix 
\begin{equation*}
f_{_{j}}\rightarrow f=\left( 
\begin{array}{cc}
\lambda & 0 \\ 
0 & \frac{1}{\lambda }%
\end{array}%
\right) \text{ and }\beta _{j}\rightarrow \beta =\Big(\lambda -\frac{1}{%
\lambda }\Big)^{2}.
\end{equation*}

It remains to show that the order $2$ elements $g_{j}$ also converge. Note
first that $\gamma_j = -\beta_j b_jc_j$ and so we may assume that $%
b_jc_j\neq 0$. As in the previous case the matrix for $g_{_{j}}$ can be
written as 
\begin{equation*}
g_{_{j}}=\Big( 
\begin{array}{cc}
a_{_{j}} & -\frac{1+a_{_{j}}{}^{2}}{c_{_{j}}} \\ 
c_{_{j}} & -a_{_{j}}%
\end{array}%
\Big) .
\end{equation*}
so that $a_j\neq \pm i$ and $1+a_j^2\neq 0$. Consider the conjugacy of the
group $\langle f_{_{j}},g_{j}\rangle $ by the following diagonal matrix, 
\begin{equation*}
\phi _{_{j}}=\left( 
\begin{array}{cc}
i\, \sqrt{\frac{c_{j}}{1+a_{_{j}}^{2}}} & 0 \\ 
0 & -i\, \sqrt{\frac{1+a_{_{j}}^{2}}{c_{_{j}}}}%
\end{array}%
\right),
\end{equation*}%
and then $g_{j}$ (replaced by as the conjugacy $\phi _{_{j}}g_{_{j}}\phi
_{_{j}}^{-1}:$) has the form 
\begin{equation*}
g_{_{j}}=\left( 
\begin{array}{cc}
a_{_{j}} & 1 \\ 
-(1+a_{_{j}}^{2}) & -a_{_{j}}%
\end{array}%
\right) .
\end{equation*}%
Since $f_{_{j}}$ commutes with $\phi _{_{j}},$ $f_{_{j}}$ is left unchanged
under the conjugation.

The commutator is now 
\begin{equation*}
\lbrack f_{_{j}},g_{j}]=f_{_{j}}g_{j}f_{_{j}}^{-1}g_{j}^{-1}=\left( 
\begin{array}{cc}
-a_{_{j}}^{2}+\lambda _{j}^{2}(1+a_{_{j}}^{2}) & -a_{j}+\lambda
_{_{j}}^{2}a_{j} \\ 
\frac{a_{j}\left( 1+a_{_{j}}^{2}\right) }{\lambda _{_{j}}^{2}}%
-a_{j}(1+a_{_{j}}^{2}) & \frac{1+a_{j}{}^{2}}{\lambda _{_{j}}^{2}}%
-a_{_{j}}^{2}%
\end{array}%
\right) .
\end{equation*}

Hence the parameter $\gamma _{j}$ for the group $\left \langle
f_{j},g_{j}\right \rangle $ is 
\begin{equation}
\gamma _{j} =\frac{a_{_{j}}^{2}\left( \lambda _{_{j}}^{2}-1\right)
^{2}+\left( \lambda _{_{j}}^{2}-1\right) ^{2}}{\lambda _{_{j}}^{2}}.
\end{equation}
At this point we would like to show $\lambda_j \not \to 1$, but this can
actually happen in limits of Dehn surgeries - see the following example.
Thus we have two subcases:

\noindent \textbf{Case (b) (i)} Suppose $\lambda_j \not \to 1$.

Solving for $a_{j}$ yields%
\begin{equation*}
a_{_{j}}^{2}=\frac{\gamma _{j}\lambda _{_{j}}^{2}-\left( \lambda
_{_{j}}^{2}-1\right) ^{2}}{\left( \lambda _{_{j}}^{2}-1\right) ^{2}}
\rightarrow \frac{\gamma \lambda ^{2}-\left( \lambda ^{2}-1\right) ^{2}}{%
\left( \lambda ^{2}-1\right) ^{2}}.
\end{equation*}%
Let $a^{2}=\frac{\gamma \lambda ^{2}-\left( \lambda ^{2}-1\right) ^{2}}{%
\left( \lambda ^{2}-1\right) ^{2}},$ then $g_{j}$ converges to the element $%
g $ of $\mathrm{PSL}(2,\mathbb{C})$ represented by the matrix%
\begin{equation*}
g=\left( 
\begin{array}{cc}
a & 1 \\ 
-(1+a^{2}) & -a%
\end{array}%
\right) .
\end{equation*}
and $\gamma_j \to \gamma(f,g)=\gamma$. Applying Theorem \ref{Th: Jorgenson},
the limit group $\Gamma =\left
\langle f,\,g\right \rangle $ is Kleinian.
\medskip

\noindent \textbf{Case (b) (ii)} Suppose $\lambda_j \not \to 1$. This is
equivalent to $\beta_j\to0$. For each $j$ we may conjugate $f_j$ and $g_j$
so that $f_j$ has a fixed point at $\infty$ and $g_j$ has a fixed point at $%
0 $. Thus $f_j$ is upper triangular and $g_j$ is lower triangular. We may
further conjugate by a diagonal matrix to achieve the upper entry in the
matrix representing $f_j$ is $1$. With this normalisation we have 
\begin{eqnarray}  \label{5.2}
f_j=\left( 
\begin{array}{cc}
\lambda_j & 1 \\ 
0 & 1/\lambda_j%
\end{array}%
\right), && g_j=\left( 
\begin{array}{cc}
i & 0 \\ 
c_j & -i%
\end{array}
\right)
\end{eqnarray}
and $\lambda_j\to 1$. We calculate that 
\begin{equation}
\gamma_j=\gamma(f_j,g_j) = c_j \left(c_j+\frac{2 i \left(\lambda_j
^2-1\right)}{\lambda_j }\right).
\end{equation}
It follows that $c_j\to \pm \sqrt{\gamma}$, for one choice of sign. Thus $%
f_j\to f$ and $g_j\to g$ 
\begin{eqnarray}
f=\left( 
\begin{array}{cc}
1 & 1 \\ 
0 & 1%
\end{array}%
\right), && g=\left( 
\begin{array}{cc}
i & 0 \\ 
\sqrt{\gamma} & -i%
\end{array}
\right)
\end{eqnarray}
This group is nonelementary by Theorem \ref{Th: Jorgenson} so $\gamma \neq 0$%
. One case see this since as soon as $|\beta_j|<\frac{1}{2}$, J\o rgensen's
inequality gives $|\gamma_j|>\frac{1}{2}$. This completes the proof.
\end{proof}

\noindent \textbf{Example.} We next give an example to show the last case
can occur. That is $\beta_j\to 0$ and $\gamma_j\to \gamma$ (and $|\gamma|=1$%
). Let

\begin{eqnarray}
f=\left( 
\begin{array}{cc}
1 & 1 \\ 
0 & 1%
\end{array}%
\right), && h=\left( 
\begin{array}{cc}
1 & 0 \\ 
a_\infty & 1%
\end{array}
\right)
\end{eqnarray}
Then $\langle f, h \rangle$ is (a representation of) the two bridge figure
of eight knot group if $a_\infty=(1+i\sqrt{3})/2$. The relator in this group
is 
\begin{equation}  \label{relator}
hfh^{-1}fh = fhf^{-1}hf
\end{equation}
If we perform $(p,0)$ Dehn surgery on the complement of this knot we obtain
a Kleinian group generated by two elliptics of order $p$, $f_p$ and $h_p$.
Then up to conjugacy $f_p$ and $h_p$ have the form 
\begin{eqnarray*}
f_p=\left( 
\begin{array}{cc}
e^{i\pi/p} & 1 \\ 
0 & e^{-i\pi/p}%
\end{array}%
\right), \; \; h_p=\left( 
\begin{array}{cc}
e^{i\pi/p} & 0 \\ 
a_p & e^{-i\pi/p}%
\end{array}%
\right)
\end{eqnarray*}
Then the relator (\ref{relator}) holds if and only if 
\begin{equation*}
a_p=\frac{1}{2} \left(3-2 \cos \frac{2 \pi }{p} -\sqrt{-1-4 \cos \frac{2 \pi 
}{p} +2 \cos \frac{4 \pi }{p} }\right)
\end{equation*}
or its conjugate. Thus $\langle f_p,h_p\rangle$ is the group obtained from
this orbifold Dehn surgery. As such it is discrete and nonelementary. Then,
as $p\to \infty$, 
\begin{eqnarray*}
\beta(f_p)=\beta(h_p)=-4\sin^2\frac{\pi}{p} & \to & 0=\beta(f), \; \; 
\mathrm{and} \\
\gamma(f_p,h_p) = a_p(-2+a_p+2 \cos \frac{2 \pi }{p}) & \to & a_\infty^2 =
\gamma(f,g)
\end{eqnarray*}
As $\beta(f_p)=\beta(h_p)$ we find an involution $g_p$ such that $%
h_p=g_pf_pg_p^{-1}$ with $\gamma(f_p,g_p)= \gamma(f_p,h_p)$.

\bigskip

We next have the following lemma.

\begin{lemma}
\label{Le: const sequence beta}Let $(\gamma _{j},\beta _{j},\widetilde{\beta 
}_{j})$ be a sequence of principal characters of two-generator Kleinian
groups $\left \langle f_{j},g_{j}\right \rangle ,$ and let $(\gamma ,\beta ,%
\widetilde{\beta })$ be the principal character of two-generator group $%
\left \langle f,\,g\right \rangle .$ Suppose that $\left( \gamma _{j},\beta
_{j}\right) $ converges to $\left( \gamma ,\beta \right) $ and $f$ is not
elliptic of order $p\in \{2,3,4,6\}.$ Then $\gamma \neq 0$ and $\gamma \neq
\beta .$
\end{lemma}

\begin{proof}
Since $f$\ is not elliptic of order $p\in \{2,3,4,6\} $, $f_{j}$ is likewise
not elliptic of these orders for all but finitely many $j$. Applying Theorem %
\ref{Th: Ksubgp}, $\left \langle f_{j},g_{j}f_{j}g_{j}^{-1}\right \rangle $
is a sequence of Kleinian groups with corresponding principal character $%
\left( \gamma _{j}\left( \gamma _{j}-\beta _{j}\right) ,\beta _{j},\beta
_{j}\right) $ which converge to $(\gamma (\gamma -\beta ),\beta ,\beta ).$
Now $\left
\langle f_{j},g_{j}f_{j}g_{j}^{-1}\right \rangle $ is Kleinian
group generated by two elements of the same trace. Thus $\gamma _{j}\left(
\gamma _{j}-\beta _{j}\right) \neq 0$ and hence $\gamma _{j}\neq 0$ and $%
\gamma _{j}\neq \beta _{j}.$

Now we may apply a result by C.Cao \cite[Theorem 5.1]{Cao} which gives a
univeral lower bound on the commutator parameter for such groups. 
\begin{equation*}
\gamma _{j}(\gamma _{j}-\beta _{j})\geq 0.198.
\end{equation*}

Thus, $\lim_{j\rightarrow \infty }\gamma _{j}(\gamma _{j}-\beta _{j})=\gamma
(\gamma -\beta )\geq 0.198.$ It follows that $\gamma \neq 0$ and $\gamma
\neq \beta .$
\end{proof}

\begin{theorem}
\label{d2closed}The two complex dimensional space 
\begin{equation*}
\mathcal{D}_{2}=\{ \left( \gamma ,\, \beta \right) :(\gamma ,\beta ,\tilde{%
\beta})\text{ is the principal character of a Kleinian group for some }%
\widetilde{\beta }\}
\end{equation*}%
\ is closed in the two complex dimensional space $\mathbb{C}^{2}$ in the
usual topology.
\end{theorem}

\begin{proof}[Proof]
Suppose that $\{(\gamma _{j},\beta _{j})\}$ is a sequence in $\mathcal{D}%
_{2} $ with limit $\left( \gamma ,\, \beta \right) $ in $\mathbb{C}^{2}.$ By
definition there is a sequence $\tilde{\beta}_j$ so that $\{(\gamma
_{j},\beta _{j},\tilde{\beta} _{j})\}$ is the principal character of a
Kleinian group. The we project to the principal character $\{(\gamma
_{j},\beta _{j},-4)\}$ of $\langle f_j,\phi_j\rangle$. If infinitely many of
these groups are Kleinian, the result we seek follows from Theorem \ref{D*
closed}. After passing to a subsequence we are left with the case $\langle
f_j,\phi_j\rangle$ is elementary for every $j$.

\noindent \textbf{Case (i)}. Suppose $f_j$ is not elliptic of order $p\in
\{2,3,4,6\}$ for infinitely many $j$. Then Lemma \ref{Le: const sequence
beta} tells us that $\gamma_j\not \in \{0,\beta_j\}$ and hence $\langle
f_j,\phi_j\rangle$ is Kleinian, so this case cannot occur.

\noindent \textbf{Case (ii)}. We may now suppose, after passing to a
subsequence, that each $f_j$ is elliptic of order $p\in \{2,3,4,5,6\}$, $%
\beta(f_j)=-4\sin^2\frac{k \pi}{p}=\beta$, $(k,p)=1$. $\langle
f_j,\phi_j\rangle$ is a discrete elementary group generated by elliptics of
orders $p$ and $2$. The order of any such group is less than that of $A_5$.
If infinitely many of these groups are finite, then in fact $[f_j,\phi]$ is
elliptic of order $1,2,3,4,5$, or $6$. Thus $\gamma_j$ lies in a finite set
of values. Again, after a subsequence we have $\gamma_j=\gamma$. Then $%
\{(\gamma _{j},\beta _{j},\tilde{\beta} _{j})\}=\{(\gamma,\beta,\tilde{\beta}
_{j})\}$ and the latter group is Kleinian by hypothesis. Thus $%
(\gamma,\beta)\in D_2$.

\noindent \textbf{Case (iii)}. We are left with the situation that $\langle
f_j,\phi_j\rangle$ is generated by elliptics of orders $p$ and $2$, infinite
and elementary for each $j$. The classification tells us this only happens
for the Euclidean groups and any such group has a common fixed point. Then $%
\gamma(f_j,\phi_j)=0=\gamma(f_j,g_j)$ which is not possible as $\langle
f_j,g_j\rangle$ is a Kleinian group.
\end{proof}

\end{document}